\documentclass[12pt]{article}
\usepackage{latexsym,epsfig,stmaryrd}
\usepackage{amsthm,amssymb}
\usepackage{graphicx}
\usepackage{amsmath, amssymb, amsfonts, amsbsy, amsthm, latexsym, graphicx}

\renewcommand{\qed}{\mbox{$\Box$}}

\newcommand{\ec}{{\mathcal E}}

\renewcommand{\r}{{\bf r}}
\newcommand{\s}{{\bf s}}
\newcommand{\R}{{\mathbf R}}
\newcommand{\Z}{{\mathbf Z}}
\newcommand{\C}{{\mathbf C}}
\newcommand{\N}{{\mathbf N}}

\newcommand{\Gb}{{\mathbf G}}
\newcommand{\pf}{{\it Proof: }}
\newcommand{\eps}{{\varepsilon}}
\newcommand{\ip}[2]{\langle#1,#2\rangle}
\newcommand{\norm}[1]{\|#1\|}
\newcommand{\opnorm}[1]{\| #1\|}

\newcommand{\fc}{{\cal F}}

\newcommand{\gc}{{\cal G}}

\newcommand{\set}[1]{\{#1\}}

\newcommand{\ignore}[1]{}
\newcommand{\plim}{\operatornamewithlimits{\mbox{$p$}-\mathrm{lim}}}
\newcommand{\sqcupd}{{\dot{\sqcup}}}
\newcommand{\gggg}{{\rho}} 

\newtheorem{Theorem}{Theorem}[section]
\newtheorem{Corollary}[Theorem]{Corollary}
\newtheorem{Lemma}[Theorem]{Lemma}

\newtheorem{Definition}[Theorem]{Definition}
\newtheorem{Remark}[Theorem]{Remark}
\newtheorem{Example}[Theorem]{Example}
\newtheorem*{introlem}{Lemma \ref{l:finite}}

\begin{document}

\title{Redundancy for localized frames}
\author{\bf Radu Balan \\
\small University of Maryland, College Park, MD 20742 \\
\small {\em rvbalan@math.umd.edu}\\
\bf Pete Casazza \\
\small University of Missouri, Columbia, MO 65211\\
\small{\em casazzap@missouri.edu} \\
\bf Zeph Landau \\
\small University of California, Berkeley, CA 94720 \\
\small {\em landau@eecs.berkeley.edu}
}

\maketitle

\begin{abstract}
Redundancy is the qualitative property which makes Hilbert space frames so 
useful
in practice.  However, developing a meaningful quantitative notion of 
redundancy for infinite frames has proven elusive.  Though quantitative 
candidates for redundancy exist, the main open problem
 is whether a frame with redundancy greater than one
contains a subframe with redundancy arbitrarily close to one.  We 
will answer this question in the affirmative for $\ell^1$-localized
frames.  We then specialize our results to Gabor multi-frames with
generators in $M^1(\R^d)$, and Gabor molecules with envelopes in $W(C,l^1)$.
  As a main tool in this work, we show there
is a universal function $g(x)$ so that for every $\epsilon>0$,
every Parseval frame $\{f_i\}_{i=1}^M$ for an $N$-dimensional Hilbert
space $H_N$ has a subset of fewer than $(1+\epsilon)N$ elements
which is a frame for $H_N$ with lower frame bound 
$g(\epsilon/(2\frac{M}{N}-1))$. This work provides 
the first meaningful quantative notion of redundancy
for a large class of infinite frames.  In addition, the results
give compelling new evidence in support of a general definition
of redundancy given in \cite{bala07}.
\end{abstract}
\thanks{The first author was supported by NSF DMS 0807896, 
 the second author was supported by NSF DMS 0704216  and thanks
the American Institute of Mathematics for their continued support.}
\section{Introduction}

A basis $\{x_i \}_{i\in I_0}$ for a Hilbert space $H$ (finite or infinite) 
with an index set $I_0$ provides a decomposition of any element $x\in H$ 
as a {\em unique} linear combination of the basis elements: 
$x= \sum_{i\in I_0} c_i x_i$.  For many applications, this uniqueness of 
decomposition is the feature that makes bases such a useful structure.  
However, there are fundamental signal processing issues for which the 
uniqueness of the coefficients $\{c_i\}_{i\in I_0}$ for a given element 
$x\in H$ is not a desired quality.  These include the following two tasks:
 a) finding ways to represent elements when some of the coefficients 
$c_i$ are going to be subject to loss or noise, and b) finding ways 
to compactly represent a meaningful approximation $x' \approx x$, 
i.e. finding  an approximation  $x' =\sum_i c'_i x_i$ that has few 
non-zero coefficients.  For both these tasks, one observes that  
choosing to express $x$ in terms of a larger set $\{f_i \}_{i\in I}$ 
that is overcomplete in $H$ has potential advantages.   With this setup, 
any vector $x\in H$ can be written as $\sum_{i\in I} c_i f_i$ in many 
different ways, and this freedom is advantageous for either of the 
above tasks.  It can allow for a choice of $\{c_i \}_{i\in I}$ with 
additional structure which can be used in the first task to counter 
the  noise on the coefficients as well as
transmission losses.  This same freedom of choice of 
$\{c_i \}_{i \in I}$ yields many more candidates for a compact 
meaningful approximation $x'$ of the element $x$. 

These overcomplete sets $\{f_i\}_{i\in I}$ (with some added structure 
when $I$ is infinite) are known as {\em frames}.  They are defined as 
follows:  let $H$ be a separable Hilbert space and $I$ a countable index 
set.  A sequence $\fc = \set{f_i}_{i \in I}$ of elements of $H$ is a
{\em frame} for $H$ if there exist constants $A$, $B>0$ such that
\begin{equation}
\label{framedef}
A \, \norm{h}^2 \le \sum_{i \in I} |\ip{h}{f_i}|^2 \le B \, \norm{h}^2,
\ \ \mbox{for all $h \in H$}.
\end{equation}
The numbers $A$, $B$ are called {\em lower} and {\em upper frame bounds},
respectively. When $A=B=1$ the frame is said to be {\em Parseval}. 
The {\em frame operator} is the operator $S:H\rightarrow H$,
  $S(x)=\sum_{i\in I}\ip{x}{f_i}f_i$, which is bounded and invertible when
$\{f_i\}_{i\in I}$ is a frame.

 Frames were first introduced by Duffin and Schaeffer \cite{dusc52}
in the context of nonharmonic Fourier series, and today frames
play important roles in many applications in mathematics, science,
and engineering.
We refer to the monograph \cite{ch03-1}, or the
research-tutorial \cite{Cas00} for basic properties of frames.

Central, both theoretically and practically, to the interest in frames
has been their overcomplete nature; the strength of this
overcompleteness is the ability of a frame to express arbitrary
vectors as a linear combination in a ``redundant'' way.  
For infinite dimensional frames, quantifying 
overcompleteness or redundancy has proven to be challenging.  What has 
been missing are results that connect redundancy of a frame to the ability 
to remove large numbers of  elements from the frame and still have the 
remaining elements form a frame.  More formally,  when imagining a measure 
of redundancy for infinite frames, an essential desired property would be a 
version of the following:

\begin{enumerate}
\item[ ${\bf P_1:}$] Any frame with redundancy bigger than one would 
contain in it a
frame with redundancy arbitrarily close to one.
\end{enumerate}

In this work, we show that for two large classes of frames -- a broad
class of Gabor systems, and $l^1$ localized frames -- the density of certain sets associated to the frame, termed the {\it frame density}, has
property $P_1$.  When combined with other work, this establishes the
frame density for these classes of frames as a legitimate quantitative
definition of redundancy.  Furthermore, it provides an additional
piece of evidence in support of a more general definition of frame
redundancy given in \cite{bala07} which applies to frames even when a
notion of density is not apparent.

\subsection{Results: finding subframes of density close to 
$1$ for Gabor and localized frames.}

A well studied important class of frames are the so called Gabor
Frames.  A {\em Gabor Frame} is defined to be a frame 
 $\fc$,
generated from time-frequency shifts of a {\em generator} 
 function
$f \in L^2(\R^d)$.  Specifically, given $f\in L^2(\R^d)$ along 
 with
a subset $\Lambda \subset \R^{2d}$: 
\[ \fc=\{ f_\lambda \} _{\lambda \in \Lambda} \mbox{ where  for } 
\lambda=(\alpha,\beta), \ \ f_{\lambda}(x)=e^{2\pi i
  \ip{\alpha}{x}}f(x-\beta). \] 
The structure of the set $\Lambda$,
more specifically various measures of the density of $\Lambda$ (see
Sections \ref{Section2} and \ref{s:6}) has been crucial in the study of 
Gabor frames.

Over the last 40 years
(since H.J. Landau \cite{L} gave a density condition for Gabor
frames whose generators were certain entire functions), partial progress 
towards a quantitative notion of redundancy has occurred for 
both lattice and general Gabor 
frames.  Many works have connected essential features of the frames to
quantities related to the density $\Lambda$ of the associated set of time
and frequency shifts (See \cite{he06-1} and references therein). 
 As dynamic as these results were, they could not be used to show that 
the obvious choice for redundancy, namely the density 
of $\Lambda$,  satisfied any version of property $P_1$.

Additional results about redundancy of arbitrary frames 
or results relating to property $P_1$ for Gabor frames have remained 
elusive.  Recent work, however, 
has made significant advances in quantifying redundancy of
infinite frames.  Progress began with the work in
\cite{bacahela03,bacahela03-1,bacahela06,bacahela06-1} which examined
and explored the notion of {\it excess} of a frame, i.e. the maximal
number of frame elements that could be removed while keeping the
remaining elements a frame for the same span.  This work, however,
left open many questions about frames with infinite excess
(which include, for example, Gabor frames that are not Riesz bases).

A quantitative approach to a large class of frames with infinite
excess (including Gabor frames) was given in
\cite{bacahela06,bacahela06-1} which introduced a general notion of 
{\it localized} frames (see also \cite{gr04-1} and then \cite{fogr04} that
independently introduced a similar notion and started a seminal discussion of frame 
localization).  The notion of localization is  between two frames
$\mathcal{F} = \{f_i\}_{i \in I}$ and
$\mathcal{E} = \{e_j\}_{j \in G}$ ($G$ a discrete abelian group),
and describes the decay of the expansion of the elements of~$\mathcal{F}$
in terms of the elements of $\mathcal{E}$ via a map $a \colon I \to G$.   
With this set up, the density of the set $a(I)$ in $G$ is a crucial quantity.  
For irregular sets $a(I)$, the density of $a(I)$ in $G$ is not a single number 
but takes on different values 
depending on additional choices, related to a 
finite decomposition of $G$ and ultrafilters, which are described Section 6.   
For the purposes of this introduction, we imagine these choices have been made 
and use the term {\it frame density} to refer to the resulting density of the 
set $a(I)$ in $G$.    Among other results, \cite{bacahela06,bacahela06-1} 
shows that in the localized setting, the frame density can be used to provide 
nice quantitative measures of frames.   A weak partial result related to 
property $P_1$ was given in 
\cite{{bacahela06},{bacahela06-1}} where it was shown that for any localized 
frame $\fc$  with frame density equal to $d$ there exists an $\epsilon >0$ 
and a 
subframe of $\fc$ with corresponding frame density $d-\epsilon$.  It is 
conjectured in \cite{bacahela06} that a version of property $P_1$ should 
hold for the frame density, and that such a result would establish frame 
density as a quantitative measure of redundancy.

In this paper, we prove this conjecture:  we show that for $l^1$ localized 
frames, the frame density has property $P_1$.   We show that for any 
$0< \epsilon <1$ every $l^1$ localized frame with frame density $d>1$ has a 
subframe with frame density smaller than $1+ \epsilon$.  Precisely, 
we show (see Section \ref{Section2} for notation and definitions):

\begin{Theorem}\label{t:main}
Assume $\fc=\{f_i\,;\,i\in I\}$ is a frame for $H$,
$\ec=\{e_k\,;\,k\in G\}$ is a $l^1$-self localized frame for $H$,
with $G$ a discrete countable abelian group, $a:I\rightarrow G$ a
localization map of finite upper density so that $(\fc,a,\ec)$ is
$l^1$ localized and has finite upper density. 
Then for every $\eps>0$ there exists a subset
$J=J_\eps\subset I$ so that $D^{+}(a;J)\le 1+\eps$ and $\fc[J]=\{f_i;i\in
J\}$ is frame for $H$.
\end{Theorem}


When specialized to Gabor frames, the result reads


\begin{Theorem}\label{t:Gabor0}
Assume $\gc(g;\Lambda)$ is a Gabor frame for $L^2(\R^d)$ with $g\in M^1(\R^d)$.
Then for every $\eps>0$ there exists a subset $J_\eps\subset\Lambda$
so that $\gc(g;J_\eps)$ is a Gabor frame for $L^2(\R^d)$ and its upper
Beurling density satisfies $D_B^{+}(J_\eps) \leq 1+\eps$.
\end{Theorem}

This result admits generalizations to both Gabor multi-frame
and Gabor molecule settings (see Section \ref{sec:Gabor}).

The work hinges on a fundamental finite dimensional result that is of 
independent interest.  For Parseval frames, the result says that an 
$M$-element Parseval frame for $H_N$ (an $N$ dimensional Hilbert space) contains a
subframe of less than $(1+\epsilon)N$ elements with lower frame
bound a function of $g(\epsilon,M/N)$, where $g$ is a universal function.  
The precise statement of the general result is given in 
Lemma \ref{l:finiteremoval}:

\begin{introlem}[Finite dimensional removal] \label{l:finite}
\label{l:finiteremoval}
There exists a monotonically increasing function  $g:(0,1)\rightarrow(0,1)$
  with the
following property. For any set
$\fc=\{f_i \}_{i=1}^M$ of $M$ vectors in a Hilbert space of
dimension $N$, and for any $0<\epsilon <\frac{M}{N}-1 $
there exists a subset $\fc_{\epsilon}\subset\fc$ of
cardinality at most $(1+\epsilon)N$ so that:
\begin{equation}\label{eq:finite}
S_{\fc_\eps}
\geq g\left( \textstyle{\tiny \frac{\epsilon}{2\frac{M}{N}-1} } \right)
S_{\fc}
\end{equation}
where $S_{\fc}f=\sum_{f\in\fc}\ip{\cdot}{f}f$ and $S_{\fc_\eps}f=
\sum_{f\in\fc_{\epsilon}}\ip{\cdot}{f}f$ are the frame operators associated
to $\fc$ and $\fc_\eps$, respectively. 
\end{introlem}

\subsection{Consequences: Redundancy}

These results complete a nice picture of redundancy for two large classes 
of frames: a broad class of Gabor systems, and $l^1$ localized frames.
 When imagining a measure of redundancy for infinite frames, in addition to 
property $P_1$ that is the focus of this work, a wish list
of desired properties would include:

\begin{enumerate}
\item[ ${\bf P_2:}$]The redundancy of any frame for the whole space would be 
greater
than or equal to one.
\item[ ${\bf P_3:}$]The redundancy of a Riesz basis would be exactly one.
\item[ ${\bf P_4:}$] The redundancy would be additive on finite unions of 
frames.
\end{enumerate}

Combining Theorem \ref{t:Gabor} with some of the results in
\cite{bacahela06,bacahela06-1} establishes that for a large class of
Gabor frames, the density of the set $\Lambda$ is a legitimate
quantative measure of redundancy (see Theorem \ref{t:6.2}
in Section \ref{s:6} for a formal statement).

\begin{Theorem} \label{t:redundancygabor}
For a Gabor molecule with envelope in $W(C,l^1)$, the Beurling density of 
its label set satisfies the properties of redundancy specified in 
$P_1$-$P_4$.
\end{Theorem}

What about similar results to Theorem \ref{t:6.2} for localized 
frames?  In this case we fix a frame $\ec$ indexed by a countable abelian 
group and consider the class of all frames $\fc$ that are $l^1$ localized 
with respect to $\ec$.  If $\ec$ is a Riesz basis,  as in the Gabor setting
the frame density can be shown to satisfy  the four desired redundancy 
properties. If $\ec$ is a frame but not necessarily a Riesz basis, two 
of the desired properties are satisfied:

\begin{Theorem} \label{t:redundancylocal}
For frames $\fc$ that are $l^1$ localized with respect to a fixed frame $\ec$ 
indexed by a countable abelian group, the frame density of $\fc$ satisfies the 
properties $P_1$ and $P_4$.  If $\ec$ is a Riesz basis then the frame 
density satisfies all properties $P_1-P_4$.
\end{Theorem}

The significance of Theorems \ref{t:redundancygabor}
and \ref{t:redundancylocal}  
is that they provide, for the first time, quantitative notions of redundancy 
for two large classes of frames that satisfy all four of the desired 
properties listed above. 

We remark that there are at least two potentially fruitful ways to view these results.  The first is to view frame density as {\em the} measure of redundancy.  From this point of view natural questions include defining notions of density for other classes of frames and proving comparable results.   

The second point of view, which we elaborate upon here, is to view
these results in the context of the work \cite{bala07} which
quantified overcompleteness for all frames that share a common index
set.  In this context, frame density should not be thought of as
redundancy but rather as a computational tool for computing redundancy
in the class of frames treated here.  Specifically, we begin by
remarking that in contrast to the Gabor molecule case, the density of
a localized frame $\fc$ depends on the frame $\ec$ that it is
localized with respect to.  When $\ec$ is a Riesz basis, the density
is ``normalized" and as a result it satisfies two properties $P_2$ and
$P_3$ that fail to hold in the``unnormalized" case of $\ec$ an
arbitrary frame.  Even when $\ec$ is a Riesz basis, the frame density
is not an intrinsic property of the frame $\fc$ and could have
different values when localized with respect to different Riesz bases.
This dependence on the frame that $\fc$ is localized with respect to
can be viewed as problematic for an optimal definition of redundancy.
In contrast, \cite{bala07} defines an intrinsic notion of redundancy
that applies to all frames that share a common index set.  The
essential tool there was the so called {\it frame measure function}
which is a function of certain averages of $\ip{f_i}{\tilde{f}_i}$,
the inner product of the frame element with its corresponding dual
frame element $\tilde{f}_i$.  A {\em redundancy function} for infinite
frames was defined to be the reciprocal of the frame measure function.
In the case of $l^1$ localized frames this redundancy function
satisfies all properties $P_1-P_4$. 
(see Section \ref{s:7} for a more complete discussion).

\subsection{Organization}
The work is organized as follows.  We begin by reviewing the
definition of localized frames.  In Section \ref{s:3} we prove the
above mentioned fundamental finite dimensional result (Lemma
\ref{l:finite}).  We then prove a ``truncation" result which is used
later to reduce the infinite dimensional case to a sequence of finite
dimensional cases.  Section \ref{s:4} contains the proof of Theorem
1.1.  We first prove Theorem 1.1 for $\ell^1$-localized Parseval
frames and then generalize this to arbitrary $\ell^1$-localized
frames.  In Section \ref{s:5} we apply this result to Gabor
Multi-frames with generators in $M^1(\R^d)$, and Gabor molecules with
envelopes in $W(C,l^1)$ and get as a Corollary Theorem 1.2.  In
Section \ref{s:6} we formally define the frame density and prove
Theorems 1.3 and 1.4.  Finally in Section \ref{s:7} we discuss consequences
in terms of the redundancy function introduced in \cite{bala07}.

\section{Notation: localized frames}\label{Section2}

The idea of localized frames in the way it is used here,
was introduced in \cite{bacahela06}. A very similar notion of frame
localization was introduced by Gr\"{o}chenig in his seminal paper
\cite{gr04-1} and then studied further e.g. in \cite{fogr04}.
For this paper, the starting point will be a Hilbert space $H$, along with
two frames for $H$:  $\fc=\{f_i~,~i\in I\}$ indexed by the countable set
$I$, and  $\ec=\{e_k~;~k\in G\}$
  indexed by a discrete countable abelian group
$G$.  Here we will assume  $G=\Z^d\times \Z_D$
for some integers $d,D\in\N$, where $\Z_D=\{0,1,2,\ldots,D-1\}$ is the
cyclic group of size $D$.

We relate the frames $\fc$ and $\ec$ by introducing a map
$a:I\rightarrow G$ between their index sets.  Following
\cite{gr04-1,fogr04,bacahela06} we
say $(\fc,a,\ec)$  is {\em $l^{p}$ localized} if
\begin{equation}
\sum_{k\in G}\sup_{i\in I} |\ip{f_i}{e_{a(i)-k}}|^p <\infty
\end{equation}
Here $1\leq p<\infty$.

We shall denote by $\r = ( r(g) )_{g\in G} $ the localization sequence
for $\fc$ with respect to $\ec$, i.e.
\[ r(g) = \sup_{i\in I,k\in G,a(i)-k=g}|\ip{f_i}{e_k}|.\]
Thus $(\fc,a,\ec)$ is $l^{1}$ localized if and only if the localization
sequence $\r$ is in $l^1(G)$.  That is,
\begin{equation}\label{equation2}
\norm{r}_{1}=\sum_{k\in G} r(k) <\infty
\end{equation}
Similarily, the set $\ec$ is said to be {\em $l^{1}$-self localized} if
\begin{equation}
\sum_{k\in G}\sup_{g\in G}|\ip{e_{k+g}}{e_g}| <\infty
\end{equation}
In other words, $\ec$ is $l^{1}$-self localized if and only if
  $(\ec,i,\ec)$ is $l^{1}$-localized, where $i:G\rightarrow G$ is the
identity map. We denote by $\s =(s(g) )_{g\in G}$ the self-localization
sequence of $\ec$, that is $s(g)=\sup_{k,l\in G,k-l=g}|\ip{e_k}{e_l}|$.

An important quantity will be the $l^1$ norm of the tail of $\r$, namely
\begin{equation}
\label {e:delta} \Delta (R) := \sum_{|k| \geq R} r(k),
\end{equation}
and thus if $(\fc,a,\ec)$ is $l^{1}$ localized,
$\lim_{R\rightarrow\infty}\Delta(R)=0$.

The {\em upper and lower densities} of a subset $J\subset I$
with respect to the map $a:I\rightarrow G$ are defined by
\begin{eqnarray} \label{e:upper}
D^{+}(a;J) & = & \limsup_{N\rightarrow\infty}
\sup_{k\in G}\frac{|a^{-1}(B_N(k))\cap J|}{|B_N(0)|} \\ \label{e:lower}
D^{-}(a;J) & = & \liminf_{N\rightarrow\infty}
\inf_{k\in G}\frac{|a^{-1}(B_N(k))\cap J|}{|B_N(0)|}
\end{eqnarray}
where $B_N(k)=\{g\in G\,;\,|g-k|\leq N\}$ is the box of radius $N$ and
center $k$ in $G$, and $|Q|$ denotes the
number of elements in the set $Q$.
Note that $|B_N(k)|=|B_N(k')|$ for all $k, k'\in G$ and $N>0$.
When $J=I$ we simply call $D^{\pm}(a;I)$ the {\it densities of $I$}, or the 
{\it densities of
the map $a$}, and we denote them by  $D^{\pm}(I)$ or $D^{\pm}(a)$.  
 The
map $a$ (or, equivalently, the set $I$) is said to have {\em finite
   upper density} if $D^{+}(I)<\infty$. As proved in Lemma 2 of
\cite{bacahela06},
if $a$ has finite upper density, then there is $K_a\geq 1$ so
that
\begin{equation}
\label{eq:Ka}
|a^{-1}(B_N(k))|\leq K_a |B_N(0)|
\end{equation}
for all $k\in G$ and $N>0$.
The finiteness of upper density is achieved when frame
vectors have norms uniformly bounded away from zero (see
Theorem 4 of \cite{bacahela06}).

%
%

\section{Two important lemmas} \label{s:3}

In this section we will prove two lemmas (Lemma \ref{l:finiteremoval}
and Lemma \ref{l:truncation}) that will be the essential ingredients 
for the proof of 
the main result (Theorem \ref{t:main}).

\subsection{Finite dimensional removal}

Here we consider the following: a finite frame $\fc=\{ f_i \}_{i=1}^M$
of $M$ vectors on an $N$ dimensional space $H$.  We are interested in
finding a subset of $\fc$ of small size that remains a frame for $H$. As
the following example illustrates, if we insist that the subset be of
size exactly $N$, we can always find a subframe, however the lower frame
bound can be very poor.

\begin{Example}\label{example1}
Denote by $\{e_1, \dots , e_N\}$ an orthonormal basis for $H_N$.  Let
$\fc$ consist of $\{e_1, \dots e_{N-1} \}$ along with $N$ copies of
$\frac{1}{\sqrt{N}}e_N$.  Thus $\fc$ is a Parseval frame with $M=2N-1$
elements.  However, a subframe with $N$ elements must be the set $\{e_1,
\dots e_{N-1}, \frac{1}{\sqrt{N}}e_N\}$ which has lower frame bound
$\frac{1}{N}=\frac{1}{M-N+1}$ which goes to zero as $N$ grows even though
the ratio $M/N$ stays bounded above by 2 and below by 1.5 (when $N\geq 2$).
\end{Example}

However, as we now show, if we allow the subset to be a little fraction
larger than $N$, i.e. of size $(1+ \epsilon) N$, then we are able to
find a subframe whose lower frame bound does not depend on $N$ but
rather on $M/N$ and $\epsilon$:

\begin{Lemma}[Finite dimensional removal] \label{l:finite}
\label{l:finiteremoval}
There exists a monotonically increasing function  $g:(0,1)\rightarrow(0,1)$
  with the
following property. For any set
$\fc=\{f_i \}_{i=1}^M$ of $M$ vectors in a Hilbert space $H_N$ of
dimension $N$, and for any $0<\epsilon <\frac{M}{N}-1 $
there exists a subset $\fc_{\epsilon}\subset\fc$ of
cardinality at most $(1+\epsilon)N$ so that:
\begin{equation}\label{eq:finite}
S_{\fc_\eps}
\geq g\left( \textstyle{\tiny \frac{\epsilon}{2\frac{M}{N}-1}  } \right)
S_{\fc}
\end{equation}
where $S_{\fc}f=\sum_{f\in\fc}\ip{\cdot}{f}f$ and $S_{\fc_\eps}f=
\sum_{f\in\fc_{\epsilon}}\ip{\cdot}{f}f$ are the frame operators associated
to $\fc$ and $\fc_\eps$, respectively. 
\end{Lemma}

\ignore{
The result can be restated in the following form:
\begin{Corollary}\label{c:finite}
There exists a monotonically increasing function $g:(0,1)\rightarrow(0,1)$
with the following property.  For any finite frame $\fc$ of $M$
elements in a Hilbert space $H_N$ of dimension $N$ with lower frame
bound $A$, and any $0<\epsilon <\frac{M}{N}-1$ there exists a subset
$\fc'\subset\fc$ of cardinality at most $(1+\epsilon)N$  that remains a
frame for $H_N$ and has lower frame bound
$Ag(\frac{\epsilon}{2M/N-1})$.  Moreover, if $\fc$ has upper frame bound
$B$ and $S_{\fc}$ (respectively, $S_{\fc'}$) is the frame operator for $\fc$
(respectively, $\fc^{'}$), then
\[ S_{\fc^{'}} \ge g\left ( \frac{\epsilon}{2M/N -1} \right ) 
\frac{A}{B}\cdot S_{\fc}.\]
\end{Corollary}

\begin{proof}
All of this is immediate from Lemma \ref{l:finite} except the {\em moreover}
part.  For this, we observe that $S_{\fc} \le B \cdot I$ and so
\[ g \cdot \frac{A}{B}\cdot S_{\fc} \le A \cdot g \cdot I \le S_{\fc'}.\]

\end{proof}
}

Example
\ref{example1} shows
that the reliance of the lower frame bound of the subframe
on $(\frac{M}{N})^{-1}$ is necessary in Lemma \ref{l:finite}. 

An estimate of the function  $g$ is given below.

To prove Lemma \ref{l:finiteremoval} we will use Lemma \ref{t:casazza}
 which is adapted from Theorem 4.3 of Casazza \cite{C2} 
(See Vershynin \cite{V} for a generalization of this
result which removes the assumption that the norms of the
frame vectors are bounded below.)  Recall that a family $\{f_i\}_{i\in I}$
is a {\em Riesz basic sequence} in a Hilbert space $H$ with 
({\em upper}, respectively {\em lower}) Riesz bounds
$A,B$ if for all families of scalars $\{a_i\}_{i\in I}$ we have:
\[ A \sum_{i\in I}|a_i|^2 \le \|\sum_{i\in I}a_if_i \|^2
\le B \sum_{i\in I}|a_i|^2.\]
Now, for the convenience of the reader we recall Theorem 4.3 in \cite{C2}.

\begin{Theorem}[Theorem 4.3 in Casazza, \cite{C2}]\label{t:Casazza0}
There is a function \\$\gggg(v,w,x,y,z):\R^5\rightarrow\R^{+}$ with the
following property: Let $(f_i)_{i=1}^M$ be any frame for an
$N$-dimensional Hilbert space $H_N$ with frame bounds $A,B$,
$\alpha\leq\norm{F_i}\leq\beta$, for all $1\leq i\leq M$, and let
$0<\eps  < 1$. Then there is a subset $\sigma\subset\{1,2,\ldots,M\}$,
with $|\sigma|\geq (1-\eps)N$ so that $(f_i)_{i\in\sigma}$ is a Riesz basis
for its span with Riesz basis constant $\gggg(\eps,A,B,\alpha,\beta)$.
\end{Theorem}

We remind the reader that the Riesz basis constant is the
larger between the upper Riesz basis bound, and the reciprocal
of the lower Riesz basis bound. 

\begin{Lemma} \label{t:casazza}
There is a monotonically increasing function $h:(0,1)\rightarrow (0,1)$
with the following property: Let $\{f_i\}_{i=1}^M$ be any
Parseval frame for an $N$-dimensional Hilbert space $H_N$ with
$\frac{1}{2}\leq\norm{f_i}^2$, for all $1\leq i\leq M$. Then for any
$0<\eps <1$ there is a subset $\sigma\subset\{1,2,\ldots,M\}$, with
$|\sigma|\geq (1-\eps)N$ so that $\{f_i\}_{i\in\sigma}$ is a Riesz
basis for its span with lower Riesz basis bound $h(\eps)$.
\end{Lemma}

{\bf Proof}:  The only part of this result which is not proved in
Theorem \ref{t:Casazza0} is
 that $h$ may be chosen to be monotonically increasing.   So let
$\gggg$ satisfy Theorem 4.3 in \cite{C2} and define for $0< \epsilon_0 <1$:
\[ h(\epsilon_0)  = \sup_{0<\epsilon \le \epsilon_0} 
\frac{1}{\gggg(\epsilon,1,1,\frac{1}{\sqrt{2}},1)}.\]
Then $h$ is monotonically increasing.  Let $\{f_i\}_{i=1}^M$ be any Parseval
frame for a $N$-dimensional
Hilbert space $H_N$ with $\frac{1}{2}\le \|f_i\|^2$ for all
$1\le i\le M$ and fix $0 < \epsilon <1$.   There exists a sequence
$\{\epsilon_n\}_{n=1}^{\infty}$ (not necessarily distinct) with
$0< \epsilon_n \le \epsilon$  so that
\[ h(\epsilon) = 1/\lim_{n\rightarrow \infty}
\gggg(\epsilon_n,1,1,\frac{1}{\sqrt{2}},1).
\]
By Theorem \ref{t:Casazza0}, for every $n\in \N$ there is a subset 
$\fc_n=\{f_i\}_{i\in I_n}$ of $\fc$ so that
\[ |\fc_n| \ge (1-\epsilon_n)N,\]
and $\{f_i\}_{i\in I_n}$ is a Riesz basis for its span with lower
Riesz basis bound $1/\gggg(\epsilon_n,1,1,\frac{1}{\sqrt{2}},1)$.  
Since the number of subsets of
$\fc$ is finite, there exists at least one subset $ \gc\subset \fc $
that appears infinitely often in the sequence $\{\fc_n\}_n$.  Thus
$|\gc| \geq (1-\epsilon_n)N$ and $\gc$ has a lower frame bound greater
than equal to $1/\gggg(\epsilon_n,1,1,\frac{1}{\sqrt{2}},1)$ 
for $n$ belonging to an infinite
subsequence of the positive integers.  Taking the limit along this
subsequence yields $|\gc| \geq lim_{n \rightarrow \infty} (1
-\epsilon_n) N= (1 - \epsilon)N$ and $\gc$ has a lower frame bound
greater than or equal to
\[ 1/\lim_{n\rightarrow \infty}\gggg(\epsilon_n,1,1,\frac{1}{\sqrt{2}},1) = 
h(\epsilon_0).\]
\qed

{\bf Proof of Lemma \ref{l:finite}}

%

{\bf Step 1.} We first assume that the frame $\{f_i \}_{i=1}^M$ is a 
Parseval frame
for its span $H_N$,  and
each vector satisfies $\norm{f_i}^2\leq \frac{1}{2}$.
Therefore, by embedding $H_N$ in a $M$-dimensional Hilbert space
and using Naimark's dilation theorem \cite{Cas00} (or the super-frame 
construction
\cite{bala07}),
we find an orthonormal basis
$\{e_i \} _{i=1}^M$ and a
projection $P$ of rank $N$ so that $f_i= Pe_i$.  Let $f_i' =
(1-P)e_i$. Then $\{f_i'\}_{i=1}^M$ is a Parseval frame for its span
and $\norm{f_i'}^2\geq\frac{1}{2}$.
Notice that we have for any set of coefficients $(c_i)_{i=1}^M$:
\begin{equation}
\label {e:finite1}
\sum_{i=1}^M |c_i|^2= \norm{\sum_{i=1}^M c_i e_i}^2 =
\norm{\sum _{i=1}^M c_i f_i } ^2 + \norm{\sum_{i=1}^M c_i f_i'}^2 .
\end{equation}

For a $\delta >0$ (that we will specify later), we now apply Lemma
\ref{t:casazza} to the frame $\{f_i' \}_{i=1}^M$ (not $\{f_i\}$) to get a
subset $\sigma\in \{1, \dots ,M \}$ with $|\sigma| \geq (1- \delta)
(M-N)$ such
that $\{f_j'\}_{j\in\sigma}$ is a Riesz basis for its span with lower
Riesz bound greater than or equal to $h(\delta)$.

Thus for any set of coefficients $(c_j) _{j \in\sigma}$ we have
$$\norm{\sum _{j\in\sigma} c_j f_j' }^2 \geq h(\delta)
\sum_{j\in \sigma}|c_j|^2.$$
Combining this with equation (\ref{e:finite1}) and a
choice of $(c_i)_{i=1}^{M}$ with the property that $c_i=0$ if $i
\not\in\sigma$ we have
\begin{equation}
\norm{\sum_{j\in\sigma} c_j f_j} ^2 \leq (1- h(\delta))
\sum_{j\in\sigma} |c_j|^2 .
\end{equation}
This equation is equivalent to saying that the operator $S_\sigma= \sum_
{j\in\sigma} \ip { \cdot }{ f_j} f_j \leq (1- h(\delta)) {\bf
   1}$.  Therefore, setting $J= I \backslash \sigma$, we have
\[ S_{J} = \sum_{j\in J} \ip{\cdot}{f_j}f_j={\bf 1} - S_\sigma 
\geq h(\delta) {\bf 1}.  \]
 Notice that $|J| \leq M -
(1-\delta)(M-N) = N +\delta (M-N) = (1 + \delta (\frac{M}{N} -1)) N$.
Thus any
choice of $\delta \leq \epsilon / (\frac{M}{N}-1)$ produces a set $J$ of 
cardinality $|J|\leq (1+\epsilon) N$ such that $S_J \geq h(\delta) {\bf 1}$.  
Setting $\delta= \frac{\epsilon}{2M/N-1}$ and $g=h$ gives the desired result.

{\bf Step 2.} Assume now that $\{f_i\}_{i=1}^M$ is a Parseval frame without
constraints on the norms of $f_i$. The upper frame bound 1 implies
$\norm{f_i}\leq 1$, for every $1\leq i\leq M$. Apply the previous result
to the Parseval frame $\{f_{i,1}\}_{i=1}^M \cup \{f_{i,2}\}_{i=1}^M$
where $f_{i,1}=f_{i,2}=\frac{1}{\sqrt{2}}f_i$ for every $1\leq i\leq M$.
Thus we obtain a set $J_1\subset \{1,2,\ldots,M\}\times\{1,2\}$,
$|J_1|\leq (1+\epsilon)N$,  so that
\[ \sum_{(i,k)\in J_1} \ip{\cdot}{f_{i,k}}f_{i,k} \geq
h\left(\frac{\epsilon}{\frac{2M}{N}-1}\right) {\bf 1} \]
Let
$J =\{ i~:~1\leq i\leq M,~{\rm such~that}~(i,1)\in J_1~{\rm or}~(i,2)\in
J_1\}$
Notice $|J|\leq |J_1|\leq (1+\epsilon)N$ and
\[ \sum_{i\in J}\ip{\cdot}{f_i}f_i \geq \sum_{(i,k)\in J_1}
\ip{\cdot}{f_{i,k}}f_{i,k} \geq
h\left(\frac{\epsilon}{2\frac{M}{N}-1}\right) {\bf 1} \]
which again produces the desired result with $g=h$.

{\bf Step 3.} For the general case, assume $S$ is the frame operator
asssociated to
$\{f_i\}_{i=1}^M$. Then $\{g_i:=S^{-1/2}f_i\}_{i=1}^M$ is a Parseval frame
with the same span. Applying the result of step 2 to this frame, we conclude
there exists a subset $J\subset\{1,2,\ldots,M\}$
of cardinality $|J|\leq (1+\epsilon)N$ so that $\{g_i;i\in J\}$ is frame such that
\[ \sum_{i \in J} \ip{\cdot}{g_i} g_i \geq
h\left(\frac{\epsilon}{2\frac{M}{N}-1}\right) {\bf 1}. \]
It follows that
\[ \sum_{i\in J}\ip{\cdot}{f_i}f_i = S^{1/2}
\left(\sum_{i\in J}\ip{\cdot}{g_i}g_i\right)S^{1/2} \geq
h\left(\frac{\epsilon}{2\frac{M}{N}-1}\right)S
\]
which is what we needed to prove. \qed

{\bf Remark:} We provide the following estimate for the function $g$.  Let 
$$b=\left(\frac{\eps}{2}\frac{A}{B}\frac{\alpha}{\beta} \right)^2$$
and choose a natural number $m$ so that
$$ (1-b)^m \leq \eps,
$$
then a careful examination of the proof of Theorem 4.3 in \cite{C2}
combined with the better constants computed in \cite{S} yields that
\begin{equation}
\rho(\eps,A,B,\alpha,\beta) \leq \frac{1}{b^{m+2}}
\end{equation}
Then an estimate for the function $g$ is given by:
\begin{equation}
\label{gest}
g(\eps) \geq \left(\frac{\eps}{2\sqrt{2}}\right)^{2+\frac{ln(\eps)}{ln(1-
\frac{\eps^2}{8})}}\geq \left(\frac{\eps^2}{8}\right)^{1+
\frac{4}{\eps^2}ln\left(\frac{1}{\eps}\right)}.
\end{equation}

\subsection{Truncation}

In this subsection we assume $\ec$ is a $l^1$ self-localized Parseval
frame for $H$ indexed by $G$, and $(\fc,a,\ec)$ is $l^1$ localized. We let
$\r$ denote the localization sequence of $\fc$, and we let $\s$ denote
the self-localization sequence of $\ec$. Further, we denote by
\begin{equation}
\label{eq:3.3.1}
f_{i,R} = \sum_{k\in G,|k-a(i)|< R} \ip{f_i}{e_k}e_k
\end{equation}
the {\em truncated} expansion of $f_i$ with respect to $\ec$.
Clearly $f_{i,R}\rightarrow f_i$ as $R\rightarrow\infty$.   But does
this convergence imply convergence of the corresponding frame operators for 
$\{f_{i,R}\}_{i\in G}$?
  The answer is that it does as we now show.  Specifically,
for a subset $J\subset I$ we denote
  $\fc[J]=\{f_i~;~i\in J\}$ and $\fc_R[J]=\{f_{i,R}~;~
i\in J\}$. Similarly we denote by $S_J$ and $S_{R,J}$
  the frame operators associated to $\fc[J]$
  and $\fc_R[J]$, respectively.
The following Lemma shows that the truncated frames well approximate the
original frames:

\begin{Lemma}
\label{l:truncation}
Choose $R_0$ so that for all $R\ge R_0$, $\Delta (R) \le (K_a \|\s\|_1)^{-1}$
(See equation (\ref{e:delta}))
 and  let $S_J$ and $S_{R,J}$ be as
above.  Then
\begin{equation}
\label{eq:3.3.3}
\norm{S_J-S_{R,J}} \leq  E(R),
\end{equation}
where $E(R)=3K_a \Delta (R)\norm{\s}_1$.
\end{Lemma}

\pf
First denote by $T_J:H\rightarrow l^2(J)$, and $T_{R,J}:H
\rightarrow l^2(J)$ the analysis maps:
\[ T_J(x)=\{\ip{x}{f_i} \}_{i\in J} ~~,~~
  T_{R,J}(x)=\{\ip{x}{f_{i,R}} \}_{i\in J} \]
Since $\ec$ is a Parseval frame, $Q:H\rightarrow l^2(G)$, $Q(x)=\{\ip{x}{e_k}
\}_{k\in G}$ is an isometry, and
\[ \norm{T_J-T_{R,J}} = \norm{(T_J-T_{R,J})^*}
=\norm{Q(T_J-T_{R,J})^*} \]
The operator $M=Q(T_J-T_{R,J})^* : l^2(J)\rightarrow l^2(G)$ is described
by a
  matrix which we also denote by $M$. In the canonical bases of $l^2(J)$
and $l^2(G)$,
  the $(k,i)$ element of $M$ is given by
\[ M_{k,i}=\ip{f_i-f_{i,R}}{e_k} = \sum_{g\in G,|g-a(i)|\geq R}
\ip{f_i}{e_g}\ip{e_g}{e_k},
\ignore{
=\left\{
\begin{array}{rcl}
\mbox{$\ip{f_i}{e_k}$} & if & \mbox{$|k-a(i)|\geq R$}  \\
  0 & & otherwise
\end{array} \right.
}
\] and thus
\begin{equation} \label{e:M}
  |M_{k,i}| \leq  \sum_{g\in G,|g-a(i)|\geq R} r(g- a(i))s(g-k)
\end{equation}
We bound the operator norm of $M$  using Schur's criterion
\cite{RieszNagy,kadrin1}
\[ \norm{M} \leq max(\sup_{i\in J}\sum_{k\in G}|M_{k,i}|, \sup_{k\in G}
\sum_{i\in J}|M_{k,i}|) \]
It follows from (\ref{e:M}) that
\begin{eqnarray*}
\sum_{k\in G}|M_{k,i}| & \leq & \Delta(R)\norm{\s}_1 \\
\sum_{i\in J}|M_{k,i}| & \leq & K_a\Delta(R)\norm{\s}_1.
\end{eqnarray*}
Thus we obtain $\norm{M}\leq K_a\Delta(R)\norm{\s}_1$ and hence
\[ \norm{T_J-T_{R,J}} = \norm{(T_J-T_{R,J})^*} \leq
K_a \Delta(R)\norm{\s}_1 \]

It follows that
\begin{eqnarray*}
 \norm{S_J - S_{R,J}} &=& \norm{(T_J-T_{R,J})^*T_J
+(T_{R,J})^*(T_J-T_{R,J})}\\
 & \leq& (\norm{T_J} + \norm{T_{R,J}} )K_a \Delta(R)
\norm{\s}_1\\
 & \leq& 3 K_a \Delta (R)\|{\s}\|_1,
 \end{eqnarray*}
the last inequality coming from $\norm{T_J} \leq 1$ and
\[\norm{T_{R,J}} \leq \norm{T_J} + \norm{T_{R,J}- T_J} \leq 1+
K_a \Delta(R)\norm{\s}_1 \leq  2,\]
since $\Delta(R) < \frac{1}{K_a\norm{\s}_1}$, for $R> R_0$.
\qed

\section{Proof of the main result} \label{s:4}

In this section we prove the main result of the paper, Theorem \ref{t:main}.

The core of the proof is contained in subsection \ref{s:specialcase}
which proves Theorem \ref{t:main} for the special case when both $\fc$
and $\ec$ are Parseval frames.  In subsection \ref{s:extension} we
show how to generalize this special case.

We begin by giving a brief description of the argument of subsection
\ref{s:specialcase}.  Our starting point is the Parseval frame $\fc$ that
is localized with respect to another Parseval frame $\ec$.  Our goal is
to produce a subset $\fc' \subset \fc$ which is a frame for the whole
space and which has density not much larger than 1.  An outline of the
steps is as follows:

\begin{enumerate}
\item Using Lemma \ref{l:truncation}, we move from the frame $\fc$ to
   a truncated frame $\fc_R$.
\item Based on the localization geometry, we decompose $I$ as the
   union of disjoint finite boxes $Q_{N}(k)$, with $k$ taking
   values in an infinite lattice.
\item For each $k$, $\fc_R[Q_N(k)]$ is a finite dimensional frame.  We
apply Lemma \ref{l:finite} to get subsets $J_{k,N,R}$ of smaller size
such that $\fc_R[J_{k,N,R}]$ remains a frame with frame operator greater
than or equal to a small constant times the frame operator for
$\fc_R[Q_N(k)]$.   Thus we have constructed a set  $J= \cup _k
J_{k,N,R}$ for which $\fc_R[J]$ is a frame for the whole space.

\item We then use our choice of $R$ along with Lemma \ref{l:truncation}
to conclude that the set $\fc [J]$ is also frame for the whole space.
\item Finally, we show that our choice of  $N$ is  large enough for the
frame $\fc[J]$ to have small
   density.
\end{enumerate}

\subsection{The case when $\fc$ and $\ec$ are Parseval}
\label{s:specialcase}

In this subsection we prove the result for the special case of
Parseval frames.
\begin{Lemma}
\label{l:special}
Let $\fc=\fc[I]$ be a Parseval frame for $H$ indexed by $I$, and let
$\ec$ be a $l^1$-self localized Parseval
frame for $H$ indexed by the discrete abelian group $G=\Z^d\times\Z_D$
so that $(\fc,a,\ec)$ is $l^1$ localized with respect to a
localization map $a:I\rightarrow G$ of finite upper density.  Then for
every $\eps>0$ there exists a subset $J=J_\eps\subset I$ so that
$D^{+}(a;J)\leq 1+\eps$ and $\fc'=\fc[J]$ is frame for $H$.
\end{Lemma}

We begin by recalling some notation.  For $k\in G$, $N\in \N$,
$B_N(k)= \{g \in G : |g-k| \leq N \}$ is the elements of $G$ in the ball 
with center $k$ and radius $N$.  Define $Q_N(k)=\{i \in I:
|a(i) -k| \leq N\}= a^{-1}(B_N(k))$. Since $D^{+}(a)<\infty$, there
exists $K_a\geq 1$ so that $|a^{-1}(B_N(k))|\leq K_a |B_N(k)|$.
Recall that we assumed $\fc$ and $\ec$ are Parseval frames for $H$,
$\ec$ is $l^1$-self localized, and $(\fc,a,\ec)$ is $l^1$
localized. Denote by $\r$ the localization sequence for $(\fc,a,\ec)$,
denote by $\s$ the self-localization sequence for $\ec$, and recall
$E(R)=3K_a\norm{\s}_1\sum_{k\in G,|k|>R}r(k)$ decays to 0 as
$R\rightarrow\infty$. Let $g:(0,1)\rightarrow(0,1)$ denote the
universal function of Lemma \ref{l:finite} and let $C_\eps$ denote
the positive quantity:
\begin{equation}
\label{Ceps}
  C_\eps = g\left( \frac{\eps}{2(2K_a-1)}\right).
\end{equation}

We now fix $\eps>0$.  For the duration we will fix two large integers
$R$ and $N$ as follows.
First $R$ is chosen so that
\begin{equation}
\label{eq:R}
E(R) < \frac{C_\eps}{2(1+C_\eps)}
\end{equation}


Then $N$ is chosen to be an integer larger than $R$ so that
\begin{equation}
\label{eq:N}
(1+\frac{\eps}{2})\frac{|B_{N+R}(0)|}{|B_{N}(0)|} \leq 1+\eps .
\end{equation}
Such an $N$ exists since $|B_M(0)|=D(2M+1)^d$ for $M>D$ and thus $\lim
_{N \rightarrow \infty} \frac{|B_{N+R }(0)|}{|B_N(0)|} =1$.

{\bf Step 1.} Define $\fc_R=\{f_{i,R}\,;\,i\in I\}$ to be the truncated
frame given by Lemma \ref{l:truncation} when it is applied to $\fc$ and
the given $R$.  Let
$S_R$ be the frame operator associated to $\fc_R$.  Notice that since
$\fc$ is a Parseval frame (and hence its frame operator is ${\bf 1}$)
we have $\opnorm{I-S_R} \leq E(R)$ and consequently
\begin{equation}
\label{eq:SR}
  (1 + E(R)) {\bf 1} \geq S_R \geq (1- E(R)) {\bf 1}.
\end{equation}

{\bf Step 2.} We let $L$ be the sublattice $(2N\Z)^d\times\{0\}\subset
G$.  For each $k \in L$ and integer $M$ let $V_{M,k} = span\{
e_j\,;\,j\in B_M(k)\}$.  Notice $dim(V_{M,k})\leq |B_M(k)|$. Let
$r_{k,N,R}=dim\,span\{f_{i,R}\,;\,i\in Q_{N}(k)\}$. Since
$span\,\{f_{i,R}\,;\,i\in Q_N(k)\}\subset V_{N+R,k}$ we obtain
$r_{k,N,R}\leq |B_{N+R}(0)|$.

If $|Q_N(k)|\leq (1+\frac{\eps}{2})|B_{N+R}(0)|$ then set
$J_{k,N,R}=Q_{N}(k)$
so that
\begin{equation}
\label{eq:sub}
\sum_{i\in J_{k,N+R}}\ip{\cdot}{f_{i,R}}f_{i,R} = \sum_{i\in
Q_{N}(k)}\ip{\cdot}{f_{i,R}}f_{i,R}\geq C_\eps
\sum_{i\in Q_{N}(k)}\ip{\cdot}{f_{i,R}}f_{i,R}
\end{equation}
where $C_\eps$ is defined in (\ref{Ceps}).

Assume now that $|Q_N(k)| > (1+\frac{\eps}{2})|B_{N+R}(0)|$.  We apply
Lemma \ref{l:finite} to the set $\{f_{i,R}\,;\,i\in Q_{N}(k)\}$ (with
  $b = (1+\frac{\eps}{2})\frac{|B_{N+R}(0)|}{r_{k,N,R}}-1$ 
as the $\epsilon >0$ in the lemma) and
obtain a subset $J_{k,N,R}\subset Q_{N}(k)$ of size
$|J_{k,N,R}|\leq(1+\frac{\eps}{2})|B_{N+R}(k)|$ so that
\begin{eqnarray}
  \sum_{i\in J_{k,N.R}}\ip{\cdot}{f_{i,R}}f_{i,R} &\geq& g\left(
\frac{b}{(2|Q_N(k)|/r_{k,N,R})-1}\right)
\sum_{i\in Q_{N}(k)}\ip{\cdot}{f_{i,R}}f_{i,R} \\
&\geq& C_{\eps} \sum_{i\in Q_{N}(k)}\ip{\cdot}{f_{i,R}}f_{i,R}
\end{eqnarray}
  where the last inequality follows from the monotonicity of $g$ and the
fact that
\[ \frac{b}{2|Q_N(k)|/r_{k,N,R}-1} \geq \frac{\eps}{2(2K_a-1)}. \]

In either case
\[ |J_{k,N,R}|\leq (1+\frac{\eps}{2})|B_{N+R}(k)|\leq (1+\eps)|B_N(0)| \]
due to (\ref{eq:N}).

{\bf Step 3.}
Set
\begin{equation}
\label{eq:J}
J_{N,R}=\cup_{k\in L} J_{k,N,R}.
\end{equation}

Denote by $S_{R,N}$ the frame operator for $\{f_{i,R}\,;
\,i\in J_{N,R} \}$.  We then have
\begin{eqnarray}
S_{R,N} & = & \sum_{k\in L}\sum_{i\in J_{k,N,R}}\ip{\cdot}
{f_{i,R}}f_{i,R} \nonumber \\
  & \geq & \sum_{k\in L}C_\eps\sum_{i\in
Q_N(k)}\ip{\cdot}{f_{i,R}}f_{i,R} = C_\eps S_R \\
& \geq & C_\eps(1-E(R)){\bf 1} \label{eq:S}
\end{eqnarray}
where the last lower bound comes from (\ref{eq:SR}).
This means $\fc_{R,N}:=\{f_{i,R}\,;\,i\in J_{N,R} \}$ is frame for $H$
with lower frame bound $C_\eps(1-E(R))$.

{\bf Step 4.}
We again apply Lemma \ref{l:truncation} with $J=J_{N,R}$
to obtain that $S_J$, the frame operator associated to
$\fc[J]=\{f_i\,;\,i\in J\}$, is bounded below by
\begin{equation}
\label{eq:SJ}
  S_J \geq S_{R,N}-E(R){\bf 1} \geq \left(C_\eps(1-E(R))-E(R)\right) {\bf 1}
\geq \frac{1}{2}C_\eps {\bf 1}
\end{equation}
where the last inequality follows from (\ref{eq:R}).
This establishes that $\fc[J]$ is frame for $H$ with lower frame bound
at least $\frac{1}{2}C_\eps$.

It remains to show that $J_{N,R}$ has the desired
upper density.

{\bf Step 5.}
The upper density of $J=J_{N,R}$ is obtained as follows. First, in each
box $B_{N}(k)$, $k\in L$, we have
\begin{equation}
\frac{|a^{-1}(B_{N}(k))\cap J|}{|B_{N}(k)|}=
\frac{|J_{k,N,R}|}{|B_{N}(k)|} \leq
(1+\frac{\eps}{2})\frac{|B_{N+R}(k)|}{|B_{N}(k)|} \leq 1+\eps
\end{equation}
Then, by an additive argument one can easily derive that
\begin{equation}
\limsup_{M\rightarrow\infty}\sup_{k\in G}\frac{|a^{-1}(J)\cap
B_{M}(k)|}{|B_M(k)|}
\leq 1+\eps
\end{equation}
which means $D^{+}(a;J)\leq 1+\eps$. \qed


\subsection{Generalizing}\label{s:extension}
 We now show how to remove the constraints that
 both $\fc$ and $\ec$ are Parseval in Lemma \ref{l:special} .
We begin by outlining the
argument: starting with the frames $\fc$ and $\ec$ we  show there are
canonical Parseval frames $\fc^{\#}$ and $\ec^{\#}$ that have the same
localization properties as $\fc$ and $\ec$.  We then apply Lemma
\ref{l:special} to these frames to get a subframe of $\fc^{\#}$ that is
a frame for the whole space with the appropriate density.  Finally, we
show that the corresponding subframe of $\fc$ has the desired frame and
density properties.

A well known canonical construction (see \cite{ch03-1}) begins with an
arbitrary  frame $\fc =\{ f_i \}$ and produces the canonical Parseval frame
\begin{equation} \label{e:sharp}
\fc ^{\#} = \{f_i^{\#}= S^{-1/2}f_i \},
\end{equation}
where $S$ is the frame operator associated to $\fc$.

In our situation we have two frames $\fc=\{f_i\,;\, i\in I\}$ and
$\ec=\{e_k\,;\,k\in G\}$ along with $a: I \rightarrow G$ such that
$(\fc,a,\ec)$ is $l^{1}$-localized and $\ec$ is a $l^1$-self localized.
As in (\ref{e:sharp}) we define two Parseval frames $\fc ^{\#}$ and $\ec
^{\#}$ corresponding to $\fc$ and $\ec$ respectively.

Lemma 2.2 from \cite{fogr04} and Theorem 2 from \cite{bacahela06} can be used
to show that $\fc ^{\#}$ and $\ec^{\#}$ inherit the localization properties 
of $\fc$ and $\ec$, namely

\begin{Lemma}
\label{l:normalization}
Given $\fc^{\#}$ and $\ec^{\#}$ as above,
if $(\fc,a,\ec)$ is $l^{1}$-localized and $\ec$ is $l^{1}$-self
  localized  then  $(\fc^{\#},a,\ec^{\#})$ is $l^{1}$-localized and
  $\ec^{\#}$ is $l^1$-self localized.
\end{Lemma}

{\bf Proof}

First, if $\ec$ is $l^1$-self localized then by Theorem 2,(c) in
\cite{bacahela06} it follows that $\ec^{\#}$ is $l^1$-self localized.
Furthermore, by Theorem 2, (b) in the aforementioned paper it follows
that $(\tilde{\ec})$ is $l^1$-self localized, where $\tilde{\ec}=
\{\tilde{e_k}~;~k\in G\}$ is the canonical dual of $\ec$.
This implies the existence
of a sequence $s\in l^1(G)$ so that 
\begin{equation}
\label{eq:l1:0}
|\ip{\tilde{e_k}}{\tilde{e_j}}| \leq s(k-j)~~,~~{\rm for~ all}~k,j\in G. 
\end{equation}
Next assume additionally that $(\fc,a,\ec)$ is $l^1$-localized. This
means there exists a sequence $r\in l^1(G)$ so that 
\begin{equation}
\label{eq:l1:1}
|\ip{f_i}{e_k}| \leq r(a(i)-k), ~{\rm for~ every}~i\in I~ {\rm and}~k\in G.
\end{equation} 
Since $\tilde{e_k}=\sum_{j\in G}\ip{\tilde{e_k}}{\tilde{e_j}}e_j$ it
follows that
\[ |\ip{f_i}{\tilde{e_k}}| = |\sum_{j\in G}\ip{f_i}{e_j}
\ip{\tilde{e_j}}{\tilde{e_k}}| \leq \sum_{j\in G} r(a(i)-j)s(j-k) = 
(r\star s)(a(i)-k), \]
(where $\star$ denotes convolution) and thus $(\fc,a,\tilde{\ec})$ is also $l^1$-localized. By Lemma 3
in \cite{bacahela06} it follows that $(\fc,a)$ is $l^1$-self localized.

Again Theorem 2, (b) implies now that $(\tilde{\fc},a)$ is $l^1$-self
localized. Therefore there exists a sequence $t\in l^1(G)$ so that
\begin{equation}
\label{eq:l1:2}
|\ip{\tilde{f_i}}{\tilde{f_j}}| \leq t(a(i)-a(j))~~,~~{\rm for~every}~i,j\in I.
\end{equation}

We will show that $(\fc,a)$ is $l^1$-self localized implies that
$(\fc^{\#},a)$ is $l^1$ localized with respect to $(\fc,a)$,
meaning that there exists a sequence $u\in l^1(G)$ so that
\begin{equation}
\label{eq:l1:3}
|\ip{f_i^{\#}}{f_j}| \leq u(a(i)-a(j))~~,~~{\rm for~every}~i,j\in I
\end{equation}
Let $\Gb:l^2(I)\rightarrow l^2(I)$ be the Gramm operator associated to
the frame $\fc$, $\Gb=TT^*$, where $T:H\rightarrow l^2(I)$ is the
analysis operator $T(x)=\{ \ip{x}{f_i} {\}}_{i\in I}$ and
$T^*:l^2(I)\rightarrow H$, $T^*(c)=\sum_{i\in I}c_if_i$ is the
synthesis operator. Let $\delta_i\in l^2(I)$ denote the sequence of
all zeros except for one entry 1 in the $i^{th}$ position. The set
$\{\delta_i,~i\in I\}$ is the canonical orthonormal basis of
$l^2(I)$. Since $\fc$ is a frame, $\Gb$ is a bounded operator with closed
range, and $T^*$ is surjective (onto). Let $\Gb^{\dagger}$ denote the
(Moore-Penrose) pseudoinverse of $\Gb$. Thus $P=\Gb\Gb^{\dagger}=\Gb^{\dagger}\Gb$ 
 is the
orthonormal projection onto the range of $T$ in $l^2(I)$. A simple exercise
shows that $\tilde{f_i}=T^*\Gb^{\dagger}\delta_i$, and 
$f_i^{\#}=T^*(\Gb^{\dagger})^{1/2}\delta_i$. Using the notation from Appendix A
of \cite{bacahela06}, we get $\Gb\in {\cal B}_1(I,a)$, the
algebra of operators that have $l^1$ decay. Using Lemma A.1 and then
the holomorphic calculus as in the Proof of Theorem 2 of the aforementioned 
paper, we obtain that $\Gb$ and all its powers $\Gb^q$, $q>0$
are in ${\cal B}_1(I,a)$. In particular, $\Gb^{1/2}\in {\cal B}_1(I,a)$
implying the existence of a sequence $u\in l^1(G)$ so that 
\[ |\ip{\Gb^{1/2}\delta_i}{\delta_j}|\leq u(a(i)-a(j)). \]
Then:
\[ \ip{f_i^{\#}}{f_j} = \ip{T^*(\Gb^{\dagger})^{1/2}\delta_i}{T^*\delta_j}
=\ip{\Gb(\Gb^{\dagger})^{1/2}\delta_i}{\delta_j} = 
\ip{\Gb^{1/2}\delta_i}{\delta_j} \]
which yields (\ref{eq:l1:3}).

The same proof applied to $(\ec,Id)$, where $Id$ is the identity map, 
implies that if
$(\ec,Id)$ is $l^1$-self localized then $(\ec^{\#},Id,\ec)$ is
$l^1$-localized (which is to say, equivalently, that $(\ec^{\#},Id)$ is
$l^1$-localized with respect to $(\ec,Id)$). Explicitely this means
there exists a sequence $v\in l^1(G)$ so that
\begin{equation}
\label{eq:l1:4}
|\ip{e^{\#}_k}{e_n}|\leq v(k-n)~~,~~{\rm for~every}~k,n\in G
\end{equation}

Putting together (\ref{eq:l1:0}-\ref{eq:l1:4}) we obtain:
\[ \ip{f_i^{\#}}{e_k^{\#}} = \sum_{j,l\in I}\sum_{m,n\in G} \ip{f_i^{\#}}{f_j}
\ip{\tilde{f_j}}{\tilde{f_l}}\ip{f_l}{e_m}\ip{\tilde{e_m}}{\tilde{e_n}}
\ip{e_n}{e_k^{\#}} \]
Hence
\begin{eqnarray}
|\ip{f_i^{\#}}{e_k^{\#}}| & \leq & \sum_{j,l\in I}\sum_{m,n\in G} \nonumber
u(a(i)-a(j))t(a(j)-a(l))r(a(l)-m)s(m-n)v(n-k) \\
& \leq & K_a^2 (u\star t\star r\star s\star v)(a(i)-k) \nonumber 
\end{eqnarray}
where $K_a$ is as in (\ref{eq:Ka}), and the convolution sequence 
$u\star t\star r\star s\star v\in l^1(G)$. This means 
$(\fc^{\#},a,\ec^{\#})$ is $l^1$ localized. $\qed$

We can now prove Theorem \ref{t:main}:

{\bf Proof of Theorem \ref{t:main} }

As above we let $\fc^{\#}$ and $\ec^{\#}$ be the canonical Parseval
frames associated with $\fc$ and $\ec$.  By Lemma \ref{l:normalization}
we have $(\fc^{\#}, a, \ec^{\#})$ is $l^1$ localized and $\ec^{\#}$ is
$l^1$-self localized.  Given $\eps >0$ we apply Lemma \ref{l:special} to
get a subset $J\subset I$ such that $D^{+}(a;J) \leq 1 + \eps$ and
$\fc^{\#} [J]$ is a frame for $H$.

To complete the proof, we now show that $\fc [J]$ is also a frame for
$H$.  This follows from the following lemma:

\begin{Lemma}
\label{l:subsets} Assume $\fc=\{f_i\,;\,i\in I\}$ is frame for $H$
with frame bounds $A\leq B$.  Let $\fc^{\#}$ be the canonical Parseval frame
associated to $\fc$.  If $J\subset I$ is such that $\{f_i^{\#},i\in J\}$
is frame for $H$ with bounds $A'\leq B'$, then
$\fc  [J]= \{f_i,i\in J\}$ is also frame for $H$ with bounds $AA'$ and
$BB'$.
\end{Lemma}

\pf
Let $S$ be the frame operator associated to $\fc$ and so $A {\bf 1}
\leq S \leq B{\bf 1}$.  Now we have the following operator inequality
\begin{eqnarray}
AA' {\bf 1} & \leq & A' S= S^{1/2} (A' {\bf 1}) S^{1/2} \\
&\leq &  S^{1/2} \left( \sum_{i\in J}
\ip{\cdot}{f_i^{\#}}f_i^{\#} \right) S^{1/2} \label{e:chain} \\
& \leq &S^{1/2} (B' {\bf 1}) S^{1/2} = B' S \leq BB' {\bf 1} .
\end{eqnarray}
Notice however that the frame operator for $\fc[J]$ satisfies
\[ \sum_{i\in J} \ip{\cdot}{f_i}f_i = S^{1/2} \left( \sum_{i\in J}
\ip{\cdot}{f_i^{\#}}f_i^{\#} \right) S^{1/2}. \]  Substituting this
equality into the middle term of the string of inequalities
(\ref{e:chain}) gives the desired result:
\[ AA' {\bf 1} \leq \sum_{i\in J} \ip{\cdot}{f_i}f_i  \leq BB' {\bf 1} .\]
\qed

\section{Application to Gabor Systems \label{sec:Gabor}} \label{s:5}

In this section we specialize to Gabor frames and molecules
the results obtained in previous section.

First we recall previously known results.

A (generic) {\em Gabor system} $\gc(g;\Lambda)$ generated by a
function $g\in L^2(\R^d)$ and a countable set of time-frequency points
$\Lambda\subset \R^{2d}$ is defined by
\begin{equation}
\gc(g;\Lambda) = \{M_\omega T_x g~;~(x,\omega)\in\Lambda\} = \{ e^{2\pi i
\ip{\omega}{t}} g(t-x)~;~(x,\omega)\in\Lambda\}.
\end{equation}
 In general we allow $\Lambda$
to be an irregular set of time-frequency points. 

A {\em Gabor multi-system} $\gc(g^1,\ldots,g^n;\Lambda^1,\ldots,\Lambda^n)$
generated by $n$ functions \linebreak
$g^1,\ldots,g^n$ and $n$ sets of time-frequency
points $\Lambda^1,\ldots,\Lambda^n$ is simply the union of the corresponding
Gabor systems:
\begin{equation}
\gc(g^1,\ldots,g^n;\Lambda^1,\ldots,\Lambda^n) = \gc(g^1;\Lambda^1)
\cup \cdots\cup \gc(g^n;\Lambda^n).
\end{equation}

A {\em Gabor molecule } $\gc(\Gamma;\Lambda)$ associated to an enveloping
function $\Gamma:\R^{2d}\rightarrow\R$ and a set of time-frequency points 
$\Lambda\subset\R^{2d}$ is a countable set of functions in $L^2(\R^d)$
indexed by $\Lambda$ whose short-time Fourier transform (STFT) have a 
common envelope of concentration:
\begin{eqnarray}
\gc(\Gamma;\Lambda) & = & \{ g_{x,\omega}~;~\,g_{x,\omega}\in L^2(\R^d)~:~ \\
 & & \nonumber
|V_\gamma g_{x,\omega}(y,\xi)|\leq \Gamma(y-x,\xi-\omega)~,~\forall 
(x,\omega)\in\Lambda\,,\,\forall(y,\xi)\in\R^{2d} \}
\end{eqnarray}
where $\gamma(t)=2^{d/4} e^{-\pi \norm{t}^2}$ and 
\begin{equation}
\label{eq:stft}
V_\gamma h(y,\xi)=\int e^{-2\pi i\ip{\xi}{t}} h(t)\gamma(t-y)dt.
\end{equation}

\begin{Remark}
Note that Gabor systems (and multi-systems) are Gabor \linebreak
molecules, 
where the common localization function is the absolute value of
the short-time Fourier transform of the generating function $g$, 
$\Gamma=|V_\gamma g|$  (or the sum of absolute values of STFTs of generating 
functions $g^1,\ldots,g^n$, $\Gamma=|V_\gamma g^1|+\cdots+|V_\gamma g^n|$).
\end{Remark}

When a Gabor system, a Gabor multi-system, or a Gabor molecule, is a frame 
we shall simply call the set a Gabor frame, a Gabor multi-frame, or a
Gabor molecule frame, respectively.

In this section the reference frame $\ec$ is going to be the Gabor frame
$\ec=\gc(\gamma;\alpha\Z^d\times\beta\Z^d)$ where $\gamma$ is the
Gaussian window $\gamma(t)=2^{d/4}e^{-\pi \norm{t}^2}$ normalized so
that its $L^2(\R^d)$ norm is one, and $\alpha,\beta>0$ are chosen so
that $\alpha\beta<1$. As is well known (see \cite{ly92,se92-1,Sewa92}),
for every such $\alpha$ and $\beta$,
$\gc(\gamma;\alpha\Z^d\times\beta\Z^d)$ is a frame for $L^2(\R^d)$.

The localization property introduced in Section 2 turns out to be equivalent
to a joint concentration in both time and frequency of the generator(s)
of a Gabor (multi-)system, or of the envelope of a Gabor molecule. The most
natural measures of concentration are given by norms of the
{\em modulation spaces}, which are Banach spaces invented and extensively
studied by Feichtinger, with some of the main references being 
\cite{fe81-3,fe89-1,fegr89,fegr89-1,fegr97}. For a detailed development
of the theory of modulation spaces and their weighted counterparts, we
refer to the original literature mentioned above and to 
\cite[Chapters~11--13]{gr01}.

For our purpose, two Banach spaces are sufficient: 
the modulation space $M^1$ and the {\em Wiener amalgam space $W(C,l^1)$}.
\begin{Definition} The {\em modulation space $M^1(\R^d)$} 
(also known as the {\em Feichtinger algebra $S_0$}) is the Banach space
consisting of all functions $f$ of $L^2(\R^d)$ so that
\begin{equation}
\label{eq:M1}
\norm{f}_{M^1}:=\norm{V_\gamma f}_{L^1} = \int\int_{\R^{2d}}
|V_\gamma f(x,\omega)|
dxd\omega~<~\infty
\end{equation}
\end{Definition}

\begin{Definition} The {\em Wiener amalgam space $W(C,l^1)$} over $\R^{n}$
is the Banach space consisting of continuous functions 
$F:\R^{n}\rightarrow\C$ so that
\begin{equation}
\label{eq:W1}
\norm{F}_{W(C,l^1)} := \sum_{k\in\Z^{n}}\sup_{t\in[0,1]^n}|F(k+t)|~<~\infty
\end{equation}
\end{Definition}

Note the Banach algebra $M^1(\R^d)$ is invariant under Fourier transform 
and is closed under both pointwise multiplication and convolution.
Furthermore, a function $f\in M^1(\R^d)$ if and only if 
$V_\gamma f\in W(C,l^1)$ over $\R^{2d}$. In particular 
the Gaussian window $\gamma\in M^1(\R^d)$.

Consider now a Gabor molecule $\gc(\Gamma;\Lambda)$ and define the localization
map $a:\Lambda\rightarrow \alpha\Z^d\times\beta\Z^d$ via $a(x,\omega)=
\left(\alpha\lfloor\frac{1}{\alpha}x\rfloor,
\beta\lfloor \frac{1}{\beta}\omega \rfloor\right)$,
 where $\lfloor\cdot\rfloor$ acts componentwise, and on each component,
 $\lfloor b\rfloor$ denotes the largest integer smaller than or equal to $b$.

For any set $J\subset\R^{2d}$, the {\em Beurling upper and lower 
density} are defined by
\begin{eqnarray}
D_B^{+}(J) & = & limsup_{N\rightarrow\infty}\sup_{z\in\R^{2d}}\frac{|\{
\lambda\in J\,:\,|\lambda-z|\leq N\}|}{(2N)^{2d}} \\
D_B^{-}(J) & = & liminf_{N\rightarrow\infty}\inf_{z\in\R^{2d}}\frac{|\{
\lambda\in J\,:\,|\lambda-z|\leq N\}|}{(2N)^{2d}}
\end{eqnarray}
The relationship between the upper and lower densities of a subset 
$J\subset\Lambda$ and the corresponding Beurling densities are given by
 (see equation (2.4) in \cite{bacahela06-1}):
\begin{eqnarray}
D^{+}(a;J) & = & (\alpha\beta)^d D_B^{+}(J) \label{eq:Dplus} \\
D^{-}(a;J) & = & (\alpha\beta)^d D_B^{-}(J)
\end{eqnarray}

We are now ready to state the main results of this section from
which Theorem \ref{t:Gabor0} follows as a Corollary:
\begin{Theorem}\label{t:G1}
Assume 
$\gc(\Gamma;\Lambda)=\{g_\lambda~;~\lambda\in\Lambda\}$ is a Gabor molecule
that is frame for $L^2(\R^d)$ with envelope $\Gamma\in W(C,l^1)$.
Then for any $\eps>0$ there exists a subset $J_\eps\subset\Lambda$
so that $\gc(\Gamma;J_\eps)=\{g_\lambda~;~\lambda\in J\}$ is frame
for $L^2(\R^d)$ and $D_B^{+}(J_\eps)\leq 1+\eps$.
\end{Theorem}

\begin{Theorem}\label{t:Gabor}
Assume $\gc(g^1,\ldots,g^n;\Lambda^1,\ldots,\Lambda^n)$ is a Gabor 
multi-frame for $L^2(\R^d)$ so that $g^1,\ldots,g^n\in M^1(\R^d)$.
Then for every $\eps>0$ there are subsets $J^1_\eps\subset\Lambda^1$,
...,$J^n_\eps\subset\Lambda^n$, so that $\gc(g^1,\ldots,g^n;J^1_\eps,\ldots,
J^n_\eps)$ is a Gabor multi-frame for $L^2(\R^d)$ and 
$D_B^{+}(J^1_\eps \cup\cdots J^n_\eps) \leq 1+\eps$.
\end{Theorem}

{\bf Proof of theorem \ref{t:G1} }
Fix $0<\eps\leq\frac{1}{2}$. Choose $\alpha,\beta>0$ so that
$(\alpha\beta)^d=1-\frac{\eps}{2}$.

First by Theorem 2.d in \cite{bacahela06-1}, it follows that
$(\gc(\gamma,\alpha\Z^d\times\beta\Z^d),i)$ is a $l^1$-self-localized
frame for $L^2(\R^d)$.

Then by Theorem 8.a in \cite{bacahela06-1} it follows that 
$(\gc(\Gamma;\Lambda),a,\gc(\gamma,\alpha\Z^d\times\beta\Z^d))$
is $l^1$-localized. Furthermore, by Theorem 9.a from the same reference,
the Beurling upper density of $\Lambda$ must be finite, hence
$D^{+}(a)<\infty$.

Thus the hypotheses of Theorem \ref{t:main} are satisfied and one can
find a subset $J_{\eps}\subset\Lambda$ so that $D^{+}(a;J_{\eps})\leq
1+\frac{\eps}{4}$. Using \ref{eq:Dplus},
\[ D^{+}_B(J_{\eps}) = \frac{D^{+}(a;J_\eps)}{(\alpha\beta)^d}\leq
\frac{1+\frac{\eps}{4}}{1-\frac{\eps}{2}}\leq 1+\eps \]
which is what we needed to prove. \qed

{\bf Proof of Theorem \ref{t:Gabor}}

First note that
$\gc(g^1,\ldots,g^n;\Lambda^1,\ldots,\Lambda^n)$ is a Gabor molecule
with envelope $\Gamma=|V_\gamma g^1|+\cdots+ |V_\gamma g^n|$. Since each
$g^1,\ldots,g^n\in M^1(\R^d)$ we obtain $\Gamma\in W(C,l^1)$ and the
conclusion follows from Theorem \ref{t:G1}. \qed

In a private communication, K. Gr\"{o}chenig pointed out to us 
that the Theorem \ref{t:Gabor} yields the following corollary:
\begin{Corollary}
\label{c:Charly}
For every $g\in M^1(\R^d)$ and $\eps>0$ there exists a countable
subset $\Lambda_{\eps,g}$ of $\R^{2d}$ with Beurling densities
$1\leq D^{-}_B(\Lambda_{\eps,g})\leq D^{+}_B(\Lambda_{\eps,g})\leq 1+\eps$
so that $\gc(g;\Lambda_{\eps,g})$ is frame for $L^2(\R^d)$.
\end{Corollary}

{\bf Proof}

Let $g$ and $\eps$ be as in hypothesis.
The general theory of coorbit spaces (\cite{fegr89,fegr89-1}
 and in particular Theorem 1 in \cite{fegr92-2}) implies that
there exists a sufficiently dense
lattice $\Sigma=\alpha \Z^{2d}$ of the phase space $\R^{2d}$ so that
$\gc(g;\alpha\Z^{2d})$ is frame for $L^2(\R^d)$. Next, Theorem \ref{t:Gabor}
implies there exists a subset $\Lambda_{\eps,g}\subset\alpha\Z^{2d}$ so that
$\gc(g;\Lambda_{\eps,g})$ remains frame for $L^2(\R^d)$ and its
upper Beurling density is bounded by $D^{+}_B(\Lambda_{\eps,g})\leq 1+\eps$.
Its lower Beurling density must be at least 1 by the general results
of irregular Gabor frames (see, e.g. \cite{chdehe99}). \qed

\section{Frame density and the proofs of Theorems \ref{t:redundancygabor} and 
\ref{t:redundancylocal} } \label{s:6}

The results presented so far have involved only lower and upper
densities: $D^{\pm}(a;I)$ in the $l^1$ localized setting, and
$D^{\pm}_B( \Lambda)$ in the Gabor setting.  These lower and upper
densities are only the extremes of the possible densities that we
could naturally assign to $I$ with respect to $a$.  In particular,
instead of taking the infimum or supremum over all possible centers as
in (\ref{e:upper}),(\ref{e:lower}) we could choose one specific sequence
of centers, and instead of computing the liminf or limsup we could
consider the limit with respect to some ultrafilter.  The different
possible choices of ultrafilters and sequences of centers gives us a
natural collection of definitions of density.  
\begin{Definition}
For a free ultrafilter $p$
and a sequence of centers $(k_n)_{n\geq 0}$
 chosen in $G$ define the
frame {\em density} to be:
\begin{equation}
\label{eq:density-p}
D(p;J;a;(k_n)_{n\geq 0}) = \plim_n \frac{|a^{-1}(B_n(k_n))\cap J|}{|B_n(0)|}.
\end{equation}
with $a:I\rightarrow G$ and $J\subset I$.
\end{Definition}
We shall denote the set of free ultrafilters $\N^{*}$  (see \cite{HS98} for more details on ultrafilters).

\begin{Definition}
For Gabor sets $(g,\Lambda)$ or Gabor molecules $\gc(\Gamma;\Lambda)$
the {\em Beurling density} of label set $\Lambda$ with respect to a
sequence of centers $(k_n)_{n\geq 0}$ and a free ultrafilter $p\in\N^*$ is
given by
\begin{equation}
D_B(p,\Lambda;(k_n)_{n\geq 0})=  \plim_n \frac{|\Lambda_n|}{(2n)^{2d}},
\end{equation}
where $\Lambda_n= \{ \lambda \in \Lambda: |\lambda -k_n| \leq n\}$.
\end{Definition}
Fore more details regarding this type of density we refer
the reader to \cite{bacahela06}. 

With these definitions, density of a set is no longer a single value
but rather a collection of values, one for each choice of centers
$k_n$ and ultrafilter $p$.  We note that all these values lie between
the upper and lower density and thus in the case where these are
equal, all these values are the same.

From here on, we fix a choice of centers $(k_n)_{n\geq 0}$ in $G$.  
Thus the frame density becomes a function $D(p,J,a)$, or
 $D(p,J)$ when the localization map $a$ is implicit. Similarly,
 the Beurling density becomes a function $D_B(p,\Lambda)$.

With these definitions, we prove the precise version of Theorems
\ref{t:redundancygabor} and \ref{t:redundancylocal}; the proofs are
straightforward consequences of the results proved here and in
\cite{bacahela06,bacahela06-1}.

\begin{Theorem}
\label{t6.1}
Assume frames $\fc=\{f_i;i\in I\}$, $\fc_1=\{f_i^1;i\in I_1\}$,
 $\fc_2=\{ f_i^2;i\in I_2 \}$ for the same Hilbert space $H$ 
are $l^1$ localized with respect to a
frame $\ec$ indexed by the countable abelian group $G$, with
 $a:I\rightarrow G$, $a_1:I_1\rightarrow G$, $a_2:I_2\rightarrow G$ being
the localization maps all of finite upper density.

1. For every $\eps>0$ there exists a subset $J_\eps\subset I$ such
that $\fc[J_\eps]=\{f_i;i\in J_\eps\}$ is frame for $H$, and $D(p,J_\eps)
\leq 1+\eps$ for all $p\in\N^*$. 

2. If $\ec$ is a Riesz basis for $H$, then $D(p,I,a)\geq 1$
 for all $p\in\N^*$.

3. If both $\fc$ and $\ec$ are Riesz bases for $H$, then 
$D(p,I,a)=1$ for all $p\in\N^*$.

4. Denote by $\fc'=\fc_1 \sqcupd \fc_2$ the disjoint union of the
two frames. Let $I'=I_1\sqcupd I_2$ and set $a':I'\rightarrow G$ 
the localization map of $\fc'$, defined by $a'(i)=a_1(i)$ if 
$i\in I_1$, and $a'(i)=a_2(i)$ if $i\in I_2$. 
Then $D(p,I',a')=D(p,I_1,a_1)+D(p,I_2,a_2)$.

\end{Theorem}

\pf

1. This comes directly from Theorem \ref{t:main} since $D(p,J_\eps)
\leq D^+(J_\eps)$.

2. $l^1$ localization implies $l^2$ localization, which in
turn implies $l^2$-column and $l^2$-row decay (Theorem 1.g in
 \cite{bacahela06}), which next implies strong HAP (Theorem 1.a in
same) and weak HAP (Theorem 1.e), and finally  that $D^{-}(I)\geq 1$ 
(Theorem 3.a in same). Consequently $D(p,I,a)\geq D^{-}(I)\geq 1$.

3. If both $\fc$ and $\ec$ are Riesz bases then $l^1$ localization
implies also weak dual HAP (see again Theorem 1 in \cite{bacahela06})
which in turn implies $D^{+}(I)\leq 1$ (Theorem 3.b in same). 
Hence $D(p,I,a)=1$ for all $p\in\N^*$.

4. The assertion comes from
\[ \frac{|a'^{-1}(B_n(k_n))|}{|B_n(0)|} = \frac{|a_1^{-1}(B_n(k_n))|}{|B_n(0)|}
+ \frac{|a_2^{-1}(B_n(k_n))|}{|B_n(0)|} \]
and the fact that $\plim$ is linear.

\qed

\begin{Theorem}
\label{t:6.2}
Assume $\gc(\Gamma;\Lambda)$, $\gc(\Gamma_1;\Lambda_1)$ and
$\gc(\Gamma_2;\Lambda_2)$ are Gabor molecules with envelopes in
$W(,C,l^1)$.  Then:

1. If $\gc(\Gamma;\Lambda)$ is frame for $L^2(\R^d)$ then for 
every $\eps>0$ there is a subset $J_\eps\subset\Lambda$
such that $\gc(\Gamma;J_\eps)$ is frame for $L^2(\R^d)$ and
 $D_B(p,J_\eps)\leq 1+\eps$ for every $p\in\N^*$.

2. If $\gc(\Gamma;\Lambda)$ is frame for $L^2(\R^d)$ then
$D(p,\Lambda)\geq 1$ for all $p\in\N^*$.

3. If $\gc(\Gamma;\Lambda)$ is a Riesz basis then $D(p,\Lambda)=1$
for all $p\in\N^*$.

4. Denote by $\gc'=\gc(\Gamma_1;\Lambda_1)\sqcupd \gc(\Gamma_2;\Lambda_2)$
the disjoint union of the two Gabor molecules. Then $\gc'$ is also
a Gabor molecule with envelope $\Gamma'=\Gamma_1+\Gamma_2$ and label set
$\Lambda'=\Lambda_1\sqcupd\Lambda_2$. Furthermore
\[ D_B(p,\Lambda')=D_B(p,\Lambda_1)+D_B(p,\Lambda_2) \]
\end{Theorem}

\pf

1. This comes directly from Theorem \ref{t:G1} since $D_B(p,J_\eps)\leq
D_B^{+}(J_\eps)$ for every $p\in\N^*$.

2. and 3. are consequences of Theorem 9(a) and (b) in \cite{bacahela06-1}
since $W(C,l^1)\subset W(C,l^2)$.

4. The statement is a direct consequence of
\[ |\Lambda'\cap B_n(k_n)| = |\Lambda_1\cap B_n(k_n)| + |\Lambda_2\cap
b_n(k_n)| \]
and linearity of p-limits.

\qed

\begin{Remark}
Theorem 9 in \cite{bacahela06-1} implies that, in the
more general case when the envelope is in $W(C,l^2)$, the density
of that Gabor molecule satisfies the properties of redundancy 
specified in $P_2$-$P_4$, that are 2.-4. in Theorem \ref{t:6.2}. 
\end{Remark}

\section{Consequences for the redundancy function} \label{s:7}

A quantification of overcompleteness for all frames that share a
common index set was given in \cite{bala07} and included a general
definition for frame redundancy.  Here we extract the relevant
definitions and results for our setting.

The basic objects are a countable index set $I$ together with
a sequence of finite subsets $(I_n)_{n\geq 0}$ that covers $I$, that is
$\cup_{n\geq 0}I_n = I$. For a subset $J\subset I$, the induced sequence
of subsets $(J_n)_{n\geq 0}$ is given simply by $J_n= J\cap I_n$.

To any frame $\fc$ indexed by $I$, $\fc=\{f_i\}_{i\in I}$, we associate
the following  {\em redundancy function}:
\begin{equation}
\label{eq:RI}
R:\N^*\rightarrow \R\cup\{\infty\}~~,~~
R(p;\fc,(I_n)_n)=\frac{1}{\plim_n \frac{1}{|I_n|}\sum_{i\in I_n}\ip{f_i}{
\tilde{f_i}}}~~,~~\forall p\in\N^*
\end{equation}
where $\tilde{f_i}=S^{-1}f_i$ are the canonical dual frame vectors,
and $\N^*$ denotes the compact space of free ultrafilters (see
\cite{bala07} for definitions).  
The limit with respect to ultrafilter $p$ is
always well-defined for bounded sequences, and since $0\leq
\ip{f_i}{\tilde{f_i}}\leq 1$ it follows the denominator in
(\ref{eq:RI}) is a real number between $0$ and $1$.

If the sequence of finite subsets is given by the context, we use
$R(p;\fc)$ to denote the redundancy function.

For Gabor frames $(f;\Lambda)$, the sequence of finite subsets
$(\Lambda_n)_{n\geq 0}$ is defined by a sequence of centers
$(k_n)_{n\geq 0}$ through $\Lambda_n = \{\lambda\in
\Lambda\,;\,|\lambda-k_n| \leq n \}$.  Then the redundancy function
(\ref{eq:RI}) becomes:
\begin{equation}
\label{eq:RR}
R:\N^*\rightarrow \R\cup\{\infty\}~,~R(p)=\frac{1}{\plim_{n}
\frac{1}{|\Lambda_n|}\sum_{\lambda\in\Lambda_n}\ip{f_\lambda}{\tilde{f_\lambda}} }.
\end{equation}

As proved in \cite{bacahela06-1}, in the case of Gabor frames, the redundancy function coincides with  the density of the label set:
\begin{Theorem}[Theorem 3(b) in \cite{bacahela06-1}]\label{t:G2}
Assume $\gc=(g;\Lambda)$ is a Gabor frame in $L^2(\R^d)$. Then for any 
sequence of centers $(k_n)_{n\geq 0}$ in $\R^{2d}$ and free ultrafilter
$p\in\N^*$,
\begin{equation}
\label{eq:Rgabor}
R(p;\gc) = D(p;\Lambda)
\end{equation}
\end{Theorem}

For a $l^1$-localized frame $(\fc,a,\ec)$ both $\fc$ and $\ec$ have their
own redundancy function.  Suppose we choose the sequences of finite subsets to be compatible with $a$ in the following way:  we choose a sequence
of centers $(k_n)_{n\geq 0}$ in $G$ and use the subsets $B_n(k_n) \subset G$ to define the redundancy function of $\ec$ and 
 $I_n=a^{-1}(B_n(k_n)) \subset I$ to define the redundancy function of $\fc$:
 \begin{eqnarray}
R(p;\fc) & = & \frac{1}{\plim_n \frac{1}{|I_n|}
\sum_{i \in I_n} \ip{f_i}{\tilde{f_i}}} \\
R(p;\ec) & = & \frac{1}{\plim_n \frac{1}{|B_n(k_n)|} \sum_{j\in B_n(k_n)}
 \ip{e_j}{\tilde{e_j}}}
\end{eqnarray}

There is a simple and important
relation between the two redundancies and the density of the map $a$:
\begin{Theorem}[Theorem 5,(b) in \cite{bacahela06}] \label{t:dens-red}
Assume $(\fc,a,\ec)$ is $l^2$-localized and has finite upper density. Then
\begin{equation}\label{eq:d-r}
R(p;\fc)=D(p,a)R(p;\ec)
\end{equation}
for all $p\in\N^*$.
\end{Theorem}
With these results in place, the main results of this work, Theorem
\ref{t:main} and \ref{t:Gabor0}, imply that a version of $P_1$ holds true for
the redundancy function of $l^1$ localized frames and Gabor frames. Specifically

\begin{Theorem}\label{t:l1-redund}
Assume $\fc=\{f_i\,;\,i\in I\}$ is a frame for $H$,
$\ec=\{e_k\,;\,k\in G\}$ is a $l^1$-self localized frame for $H$,
with $G$ a discrete countable abelian group, $a:I\rightarrow G$ a
localization map of finite upper density so that $(\fc,a,\ec)$ is
$l^1$ localized. Then for every $\eps>0$ there exists a subset
$J=J_\eps\subset I$ so that $\fc[J]=\{f_i;i\in
J\}$ is frame for $H$ and 
\begin{equation}
\label{eq:R1}
R(p;\fc[J]) \leq (1+\eps)R(p;\ec)
\end{equation}
for all $p\in\N^*$.
\end{Theorem}

When specialized to Gabor frames, this result reads:

\begin{Theorem}\label{t:Gabor-redund}
Assume $\gc(g;\Lambda)$ is a Gabor frame for $\L^2(\R^d)$ with
$g\in M^1(\R^d)$. Then for every $\eps>0$ there exists a subset 
$J_\eps\subset\Lambda$ so that $\gc'=\gc(g;J_\eps)$ is a Gabor frame 
for $L^2(\R^d)$ and its redundancy 
is upper bounded by $1+\eps$, 
\[ R(p;\gc')\leq 1+\eps \]
for all $p\in\N^*$.
\end{Theorem}

By construction the redundancy function satisfies properties $P_2$ and $P_3$
regardless of any localization property: for any frame $\fc$ indexed
by $I$,
\[ R(p;\fc)\geq 1~~,~~\forall p\in\N^*. \]
When $\fc$ is a Riesz basis
\[ R(p;\fc) = 1 ~~,~~\forall p\in\N^*. \]

Theorem \ref{t:dens-red} shows that in the setting of a frame 
$\fc$ that is $l^2$ localized with
respect to frame $\ec$, the redundancy function of $\fc$ is the product
of the redundancy function for $\ec$ {\it with } the frame density.
The redundancy function of \cite{bala07} is identically $1$ for any Riesz basis
and thus when $\ec$ is a Riesz basis and $\fc$ is $l^2$ localized with respect to $\ec$, the redundancy function for $\fc$  is {\em equal} to the frame density; consequently, for this case, the redundancy property satisfies the property $P_4$.
Combining all these results, the redundancy function
satisfies all four properties $P_1-P_4$ in the case
of a frame that is $l^1$ localized with respect to a family of frames 
of redundancy arbitrary close to $1$:
\begin{Theorem}
Assume $\ec_n$ be a sequence of $l^1$-self localized frames of $H$
all indexed by the discrete abelian group $G$ so that 
$liminf_n R(p,\ec_n) = 1$ for all $p\in\N^*$.
Assume $\fc=\{f_i,i\in I\}$ is a frame for $H$ and $(\fc,a,\ec_n)$ are 
all $l^1$-localized for all $n$, with respect to a localization map 
$a:I\rightarrow G$.

1. For every $\eps>0$ there is a subset $J_\eps\subset I$ so that
 $\fc[J_\eps]=\{f_i;i\in J_\eps\}$ is frame for $H$ and
$R(p;\fc[J_\eps])\leq 1+\eps$ for all $p\in\N^*$.

2. $R(p;\fc)\geq 1$, for all $p\in\N^*$.

3. If $\fc$ is a Riesz basis for $H$, then $R(p;\fc)=1$ for
all $p\in\N^*$.

4. Assume $\fc_1=\{f_i^1,i\in I\}$ and $\fc_2=\{f_i^2,i\in I\}$ are
two frames for $H$ so that $(\fc_k,a,\ec_n)$ are $l^1$-localized
for all $n$ and $k=1,2$. Then
\[ R(p;\fc_1\sqcupd \fc_2) = R(p;\fc_1) + R(p;\fc_2) \]
for all $p\in\N^*$.
\end{Theorem}

\noindent {\bf Acknowledgment:}  The authors thank Karlheinz Gr\"{o}chenig,
Chris Heil, and Roman Vershynin for their valuable contributions to the content and presentation of this work.



\begin{thebibliography}{10}

\bibitem{bacahela03}
{R}adu {B}alan, {P}eter~{G}. {C}asazza, {C}hristopher {H}eil, and {Z}eph
  {L}andau.
\newblock {D}eficits and excesses of frames.
\newblock {\em {A}dv. {C}omput. {M}ath.}, 18(2-4):93--116, 2003.

\bibitem{bacahela03-1}
{R}adu {B}alan, {P}eter~{G}. {C}asazza, {C}hristopher {H}eil, and {Z}eph
  {L}andau.
\newblock {E}xcesses of {G}abor frames.
\newblock {\em {A}ppl. {C}omput. {H}armon. {A}nal.}, 14(2):87--106, 2003.

\bibitem{bacahela06}
{R}. {B}alan, {P}.{G}. {C}asazza, {C}. {H}eil, and {Z}. {L}andau.
\newblock {D}ensity, overcompleteness, and localization of frames {I}:
  {T}heory.
\newblock {\em {J}. {F}ourier {A}nal. {A}ppl.}, 12(2):105--143, 2006.

\bibitem{bacahela06-1}
{R}. {B}alan, {P}.{G}. {C}asazza, {C}. {H}eil, and {Z}. {L}andau.
\newblock {D}ensity, overcompleteness, and localization of frames {I}{I}:
  {G}abor frames.
\newblock {\em {J}. {F}ourier {A}nal. {A}ppl.}, 12(3):307--344, 2006.

\bibitem{bala07}
{R}. {B}alan and {Z}. {L}andau.
\newblock Measure functions for frames.
\newblock {\em {J}. {F}unct. {A}nal.}, 252:630--676, 2007.

\bibitem{C2}
{P}.{G}. {C}asazza.
\newblock Local theory of frames and schauder bases for hilbert space.
\newblock {\em {I}llinois {J}our. {M}ath.}, 43:291--306, 1999.

\bibitem{Cas00}
{P}.{G}. {C}asazza.
\newblock {T}he {A}rt of {F}rame {T}heory.
\newblock {\em {T}aiwanese {J}. {M}ath.}, 4:129--201, 2000.

\bibitem{chdehe99}
O.~{C}hristensen, B.~{D}eng, and C.~{H}eil.
\newblock {D}ensity of {G}abor frames.
\newblock {\em {A}ppl. {C}omput. {H}armon. {A}nal.}, 7(3):292--304, 1999.

\bibitem{ch03-1}
{O}le {C}hristensen.
\newblock {\em {A}n {I}ntroduction to {F}rames and {R}iesz {B}ases}.
\newblock {B}irkh{\"a}user, {B}oston, 2003.

\bibitem{dusc52}
{R}.{J}. {D}uffin and {A}.{C}. {S}chaeffer.
\newblock {A} class of nonharmonic {F}ourier series.
\newblock {\em {T}rans. {A}m. {M}ath. {S}oc.}, 72:341--366, 1952.
\newblock reprinted in hewa06.

\bibitem{fe81-3}
{H}.{G}. {F}eichtinger.
\newblock {O}n a new {S}egal algebra.
\newblock {\em {M}onatsh. {M}ath.}, 92:269--289, 1981.

\bibitem{fe89-1}
{H}.{G}. {F}eichtinger.
\newblock {A}tomic characterizations of modulation spaces through {G}abor-type
  representations.
\newblock In {\em {P}roc. {C}onf. {C}onstructive {F}unction {T}heory},
  volume~19 of {\em {R}ocky {M}ountain {J}. {M}ath.}, pages 113--126, 1989.

\bibitem{fegr89}
{H}.{G}. {F}eichtinger and {K}. {G}r{\"o}chenig.
\newblock {B}anach spaces related to integrable group representations and their
  atomic decompositions, {I}.
\newblock {\em {J}. {F}unct. {A}nal.}, 86:307--340, 1989.

\bibitem{fegr89-1}
{H}.{G}. {F}eichtinger and {K}. {G}r{\"o}chenig.
\newblock {B}anach spaces related to integrable group representations and their
  atomic decompositions, {I}{I}.
\newblock {\em {M}onatsh. {M}ath.}, 108:129--148, 1989.

\bibitem{fegr92-2}{
{H}.{G}. {F}eichtinger and {K}. {G}r{\"o}chenig,
{\em {N}on-{O}rthogonal wavelet and {G}abor expansions, and group representations},
in "{W}avelets and their applications", {B}eylkin, {G}. and {C}oifman, {R}. and {D}aubechies, {I}. Eds., (1992), pp. 353-376
}

\bibitem{fegr97}
{H}.{G}. {F}eichtinger and {K}. {G}r{\"o}chenig.
\newblock {G}abor frames and time-frequency analysis of distributions.
\newblock {\em {J}. {F}unct. {A}nal.}, 146(2):464--495, 1997.

\bibitem{fogr04}
{M}. {F}ornasier and {K}. {G}r{\"o}chenig.
\newblock {I}ntrinsic localization of frames.
\newblock {\em {C}onstr. {A}pprox.}, 22(3):395--415, 2005.

\bibitem{gr01}
{K}. {G}r{\"o}chenig.
\newblock {\em {F}oundations of time-frequency analysis}.
\newblock {A}ppl. {N}umer. {H}armon. {A}nal. {B}irkh{\"a}user {B}oston,
  {B}oston, {M}{A}, 2001.

\bibitem{gr04-1}
{K}. {G}r{\"o}chenig.
\newblock {L}ocalization of {F}rames, {B}anach {F}rames, and the
  {I}nvertibility of the {F}rame {O}perator.
\newblock {\em {J}. {F}ourier {A}nal. {A}ppl.}, 10(2):105--132, 2004.

\bibitem{he06-1}
{C}. {H}eil.
\newblock {O}n the history and evolution of the density theorem for {G}abor
  frames.
\newblock Technical report, Georgia Institute of Technology, 2006.

\bibitem{HS98}
{N}. {H}indman and {D}. {S}trauss,
Algebra in the Stone-\v{C}ech Compactification,
de Gruyter Expositions in Mathematics Vol.~27,
Walter de Gruyter and Co., Berlin, 1998.

\bibitem{kadrin1}
{R}.{V}. {K}adison and {J}.{R}. {R}ingrose.
\newblock {\em {F}undamentals of the {T}heory of {O}perator {A}lgebras. {I}}.
\newblock {AMS} {G}raduate {S}tudies in {M}athematics 15, 1997.

\bibitem{L}
{H}.{J}. {L}andau.
\newblock {\em {N}ecessary density conditions for sampling and
interpolation of certain entire functions}.
\newblock {Acta Math.} 117: 37-52, 1967.

\bibitem{ly92}
{Y}u.{I}. {L}yubarskij.
\newblock {F}rames in the {B}argmann space of entire functions.
\newblock In {\em {E}ntire and subharmonic functions}, volume~11 of {\em {A}dv.
  {S}ov. {M}ath.}, pages 167--180. {A}merican {M}athematical {S}ociety
  ({A}{M}{S}), {P}rovidence, {R}{I}, 1992.

\bibitem{RieszNagy}
{F} {R}iesz and {B}.{S}. {N}agy.
\newblock {\em {F}unctional {A}nalysis.}
\newblock {D}over {P}ublications, 1990.

\bibitem{se92-1}
{K}ristian {S}eip.
\newblock {D}ensity theorems for sampling and interpolation in the
  {B}argmann-{F}ock space. {I}.
\newblock {\em {J}. {R}eine {A}ngew. {M}ath.}, 429:91--106, 1992.

\bibitem{Sewa92}
{K}ristian {S}eip and {R}obert {W}allst{\'e}n.
\newblock {D}ensity theorems for sampling and interpolation in the
  {B}argmann-{F}ock space. {I}{I}.
\newblock {\em {J}. {R}eine {A}ngew. {M}ath.}, 429:107--113, 1992.

\bibitem{S}  D.S. Spielman and N. Srivastave, {\em An elementary proof of the
restricted invertibility theorem}, preprint, 
http://arxiv.org/abs/0911.1114, Nov. 2009.

\bibitem{V}
{R}. {V}ershynin.
\newblock Subsequences of frames.
\newblock {\em {S}tudia {M}athematica}, 145:185--197, 2001.

\end{thebibliography}
\end{document}